\documentclass{amsart}
\usepackage{amsfonts,amsmath,amssymb,times,hyperref,tikz,subfigure}

\theoremstyle{plain}
\newtheorem{thm}{Theorem}[section]
\newtheorem{dfn}[thm]{Definition}
\newtheorem{cor}[thm]{Corollary}

\newtheorem{lemma}[thm]{Lemma}

\newcounter{example}

\numberwithin{equation}{section}

\DeclareMathOperator{\cl}{cl}

\newcommand{\del}{\setminus}

\newcommand{\join}{\lor}
\newcommand{\meet}{\land}

\hypersetup{colorlinks=true, linkcolor=blue, urlcolor=blue,citecolor=blue}

\title[Connectivity Gaps]{Connectivity Gaps Among Matroids With The
  Same Enumerative Invariants} \author[J.~Bonin]{Joseph E.~Bonin}
\address
{Department of Mathematics\\ The George Washington University\\
  Washington, D.C.\ 20052, USA} \email {jbonin@gwu.edu,
  kevinlong@gwmail.gwu.edu} \author[K.~Long]{Kevin Long} \date{\today}


\begin{document}

\begin{abstract}
  Many important enumerative invariants of a matroid can be obtained
  from its Tutte polynomial, and many more are determined by two
  stronger invariants, the $\mathcal{G}$-invariant and the
  configuration of the matroid.  We show that the same is not true of
  the most basic connectivity invariants.  Specifically, we show that
  for any positive integer $n$, there are pairs of matroids that have
  the same configuration (and so the same $\mathcal{G}$-invariant and
  the same Tutte polynomial) but the difference between their Tutte
  connectivities exceeds $n$, and likewise for vertical connectivity
  and branch-width.  The examples that we use to show this, which we
  construct using an operation that we introduce, are transversal
  matroids that are also positroids.
\end{abstract}

\maketitle

\section{Introduction}\label{section:introduction}

The most well-known and widely-studied enumerative invariant of a
matroid $M$ is its \emph{Tutte polynomial}, which is given by
$$T(M;x,y)=\sum_{A\subseteq E(M)}(x-1)^{r(M)-r(A)}(y-1)^{|A|-r(A)}.$$
From $T(M;x,y)$, one can deduce the number of bases of $M$, the number
of independent sets of $M$, the minimum size among the circuits of
$M$, the characteristic polynomial of $M$ (which, when $M$ is the
cycle matroid of a graph, includes the chromatic polynomial), and
much, much more.  See \cite{TomJames, Tutte} for extensive treatments
of this important invariant.  The Tutte polynomial of $M$ can be
computed from a strictly stronger enumerative invariant of $M$, the
$\mathcal{G}$-invariant, which was introduced more recently by Derksen
\cite{G-inv}.

For a matroid $M$, a subset $A$ of $E(M)$ is \emph{cyclic} if the
restriction $M|A$ has no coloops.  Equivalently, $A$ is cyclic if it
is a (possibly empty) union of circuits.  \emph{Cyclic flats} are
flats that are cyclic.  The cyclic flats of a matroid $M$ form a
lattice, which is denoted by $\mathcal{Z}(M)$.  Brylawski
\cite{affine} pointed out that a matroid is determined by its cyclic
flats and their ranks.  Thus, $T(M;x,y)$ is determined by the cyclic
flats of $M$ and their ranks.  Eberhardt \cite{config} showed that for
a matroid $M$ with no coloops, $T(M;x,y)$ is determined by
considerably less information, namely, the abstract lattice of cyclic
flats of $M$ along with the size and rank of the cyclic flat that
corresponds to each lattice element. (Extending this result to all
matroids by recording the number of coloops is completely routine.)
This abstract lattice along with the assignment of the size and rank
to each lattice element is what Eberhardt defined to be the
\emph{configuration} of $M$.  See Figure \ref{fig:configexamples} for
an example.

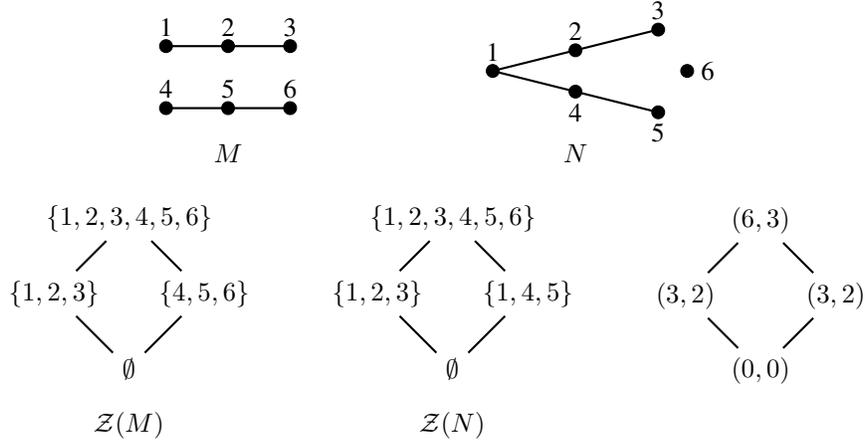
\begin{figure}[t]
  \centering

  \begin{tikzpicture}[scale=1.1]
    \draw[thick](0,0)--(1.5,0);%
    \draw[thick](0,-.75)--(1.5,-.75);%
    \filldraw (0,0) node {} circle (2.2pt);%
    \node (111) at (0,.23) {1};%
    \filldraw (.75,0) node {} circle (2.2pt);%
    \node at (.75,.23) {2};%
    \filldraw (1.5,0) node {} circle (2.2pt);%
    \node at (1.5,.23) {3};%
    \filldraw (0,-.75) node {} circle (2.2pt);%
    \node at (0,-.52) {4};%
    \filldraw (.75,-.75) node {} circle (2.2pt);%
    \node at (.75,-.52) {5};%
    \filldraw (1.5,-.75) node {} circle (2.2pt);%
    \node at (1.5,-.52) {6};%
    \node at (0.75,-1.3) {$M$};%
  \end{tikzpicture}
  \hspace{60pt}
  \begin{tikzpicture}[scale=1.1]
    \draw[thick](0,0)--(2,.5); %
    \draw[thick](0,0)--(2,-.5); %
    \filldraw (0,0) node {} circle (2.2pt); %
    \node at (0,.23) {1};%
    \filldraw (1,.25) node {} circle (2.2pt); %
    \node at (1,.48) {2}; %
    \filldraw (2,.5) node {} circle (2.2pt); %
    \node at (2,.73) {3}; %
    \filldraw (1,-.25) node {} circle (2.2pt); %
    \node at (1,-0.5) {4}; %
    \filldraw (2,-.5) node {} circle (2.2pt); %
    \node at (2,-.75) {5}; %
    \filldraw (2.35,0) node {} circle (2.2pt); %
    \node at (2.6,0) {6}; %
    \node at (1,-1) {$N$};%
  \end{tikzpicture}

  \bigskip
    
  \begin{tikzpicture}[scale = 1]
    \node (111) at (1, 3) {$\{1,2,3,4,5,6\}$};%
    \node (110) at (0, 2) {$\{1,2,3\}$};%
    \node (101) at (2, 2) {$\{4,5,6\}$};%
    \node (100) at (1, 1) {$\emptyset$};%
    \draw[thick] (111) -- (110);%
    \draw[thick] (111) -- (101);%
    \draw[thick] (110) -- (100);%
    \draw[thick] (101) -- (100); %
    \node at (1,0.25) {$\mathcal{Z}(M)$};%
  \end{tikzpicture}
  \hspace{20pt}
  \begin{tikzpicture}[scale = 1]
    \node (111) at (1, 3) {$\{1,2,3,4,5,6\}$};%
    \node (110) at (0, 2) {$\{1,2,3\}$};%
    \node (101) at (2, 2) {$\{1,4,5\}$};%
    \node (100) at (1, 1) {$\emptyset$}; %
    \draw[thick] (111) -- (110);%
    \draw[thick] (111) -- (101);%
    \draw[thick] (110) -- (100);%
    \draw[thick] (101) -- (100);%
    \node at (1,0.25) {$\mathcal{Z}(N)$};%
  \end{tikzpicture}
  \hspace{20pt}
  \begin{tikzpicture}[scale = 1]
    \node (111) at (1, 3) {$(6,3)$};%
    \node (110) at (0, 2) {$(3,2)$};%
    \node (101) at (2, 2) {$(3,2)$};%
    \node (100) at (1, 1) {$(0,0)$};%
    \draw[thick] (111) -- (110);%
    \draw[thick] (111) -- (101);%
    \draw[thick] (110) -- (100);%
    \draw[thick] (101) -- (100);%
    \node at (1,0.25) {{\color{white}(c)}};%
  \end{tikzpicture}
  \caption{Two non-isomorphic matroids, $M$ and $N$, of rank $3$,
    their lattices of cyclic flats, and the configuration that they
    share.  The size, $s$, and rank, $k$, of the cyclic flat that
    corresponds to an element in the lattice is denoted by the pair
    $(s,k)$.  }
  \label{fig:configexamples}
\end{figure}

In \cite{catdata}, Bonin and Kung showed that Eberhardt's result
extends to the $\mathcal{G}$-invariant: the $\mathcal{G}$-invariant of
a matroid with no coloops can be computed from its configuration.  (As
above, extending this to all matroids by recording the number of
coloops is routine.)  Given these results, we focus on the
configuration of a matroid.

A fundamental concept in structural matroid theory is connectivity, which
comes in many varieties, of which we focus on the three that are used
most often: Tutte connectivity, vertical connectivity, and
branch-width.  In contrast to the many enumerative invariants that are
shared by matroids that have the same configuration, we show that for
any positive integer $n$, there are pairs of matroids that have the
same configuration but the difference in their Tutte connectivities
exceeds $n$, and likewise for vertical connectivity and branch-width.
We show this in two steps: first, we give examples of small matroids
with the same configuration but different Tutte connectivities (or
vertical connectivities, or branch-widths); then, using the
$t$-expansion of a matroid, which we introduce in Section
\ref{section:texpansion}, we produce large examples, with the same
structural features, that magnify the connectivity gaps to any desired
degree.  The results on gaps in the various types of connectivity
follow from analyzing how each type of connectivity changes when we
apply $t$-expansion.

The examples that we use belong to two important classes of matroids:
transversal matroids and positroids.  While transversal matroids are a
long-studied class of matroids, positroids were introduced much more
recently, by Blum \cite{blum} (calling them base-sortable matroids)
and Postnikov \cite{post}.  A rich theory of positroids is developing
rapidly (e.g., see \cite{ARW,blum,oh,post}).  A \emph{positroid} is a
matroid $M$, say of rank $r$ and with $|E(M)|=n$, that has a matrix
representation by an $r$ by $n$ matrix over $\mathbb{R}$ having the
property that each submatrix that is formed by deleting all but $r$
columns has a nonnegative determinant.  To show that all of our
examples are transversal and are positroids, we show this is true of
our initial examples and show that both of these classes of matroids
are preserved by $t$-expansion.

We do not address whether such gaps in connectivity, among matroids
with the same configuration, can be found in other commonly-studied
classes of matroids, such as graphic or binary matroids, or any
minor-closed, well-quasi-ordered family of matroids.  This is an
interesting question for future research.

In Section \ref{section:background} we recall the required background.
The $t$-expansion of a matroid is defined in Section
\ref{section:texpansion} and its basic properties are developed.  The
results on the effect of $t$-expansion on the three types of
connectivity, and the gaps in the connectivities of matroids that have
the same configuration, are treated in Sections
\ref{section:connectivity} (for Tutte connectivity),
\ref{section:vertical} (for vertical connectivity), and
\ref{section:branch} (for branch-width).

We assume basic familiarity with matroid theory, for which we refer to
\cite{oxley}, and we follow the notation established there.  We use
$\mathbb{N}$ to denote the set of positive integers and $[n]$ to
denote the set $\{1,2,\ldots,n\}$. When discussing a matroid $M$ and
subset $X$ of $E(M)$, we use $\overline{X}$ for the complement,
$E(M)-X$, of $X$ relative to $E(M)$.

\section{Background}\label{section:background}

For a matroid $M$ and subset $X$ of $E(M)$, we have
$M=(M|X)\oplus(M\del X)$ if and only if $r(M) = r(X)+r(\overline{X})$.
A subset $X$ of $E(M)$ for which those equalities hold is a
\emph{separator} of $M$.  A separator $X$ of $M$ is \emph{nontrivial}
if $X$ is neither $\emptyset$ nor $E(M)$.  Matroids that have no
nontrivial separators are \emph{connected}.  The notions of
connectivity that we review next capture higher types of connectivity.

\subsection{Tutte connectivity and vertical connectivity}\label{ssec:tcvc}

For greater detail on all topics in this section, the reader is
referred to \cite[Chapter 8]{oxley}.  All notions of connectivity that
we discuss are defined using the \emph{connectivity function}
$\lambda_M:2^{E(M)}\to\mathbb{Z}$ of a matroid $M$, which is given by
$$\lambda_M(X)
= r_M(X)+r_M(\overline{X})-r(M)$$ for $X\subseteq E(M)$.  It follows that
$\lambda_M(X)=\lambda_M(\overline{X})$ and that
$\lambda_M(X)=\lambda_{M^*}(X)$ for all $X\subseteq E(M)$.

For $k\in\mathbb{N}$, a \emph{$k$-separation} of a matroid $M$ is a
pair $(X,\overline{X})$, with $X\subseteq E(M)$, for which
$|X|\geq k$, $|\overline{X}|\geq k$, and $\lambda_M(X)<k$.  Thus,
$(X,\overline{X})$ is a $1$-separation if and only if $X$ (and thus
$\overline{X}$) is a nontrivial separator.  For a matroid $M$ that has
a $k$-separation for at least one $k\in\mathbb{N}$, its
\emph{connectivity} or \emph{Tutte connectivity}, denoted $\tau(M)$,
is the least $k$ for which $M$ has a $k$-separation.  The only
matroids that have no $k$-separations for any $k\in \mathbb{N}$ are
the uniform matroids $U_{r,n}$ with $r\in\mathbb{N}$ and
$n\in\{2r-1,2r,2r+1\}$; for such matroids, we set $\tau(M)=\infty$.  A
matroid $M$ is \emph{$n$-connected} or \emph{Tutte $n$-connected} if
$\tau(M)\geq n$.

For $k\in\mathbb{N}$, a \emph{vertical $k$-separation} of $M$ is a
pair $(X,\overline{X})$, with $X\subseteq E(M)$, for which
$r(X)\geq k$, $r(\overline{X})\geq k$, and $\lambda_M(X)<k$.  It is
easy to show that a matroid $M$ has a vertical $k$-separation for at
least one $k\in\mathbb{N}$ if and only if $E(M)$ is a union of two
proper flats of $M$.  For such a matroid $M$, its \emph{vertical
  connectivity}, denoted $\kappa(M)$, is the least $k$ for which $M$
has a vertical $k$-separation.  If $E(M)$ is not the union of two
proper flats, we set $\kappa(M)=r(M)$.  A matroid $M$ is
\emph{vertically $n$-connected} if $\kappa(M)\geq n$.

For example, for the matroids shown in Figure
\ref{fig:configexamples}, we have $\lambda_M(\{1,2,3\})=1$ while
$\lambda_N(\{1,2,3\})=2$, and $\tau(M)=\kappa(M)=2$, while
$\tau(N)=\kappa(N)=3$.

While $\tau(M)=\tau(M^*)$, the same is not true for vertical
connectivity.

Note that if $e$ is a loop of $M$ and $|E(M)|\geq 2$, then
$(\{e\},\overline{\{e\}})$ is a $1$-separation of $M$ but not a
vertical $1$-separation.  Indeed, if $L$ is the set of loops of $M$,
then $\kappa(M)=\kappa(M\del L)$.  Thus, a matroid $M$ that has loops
may have $\kappa(M)>1$; on the other hand, for a loopless matroid $M$,
we have $\kappa(M)=1$ if and only if $M$ is disconnected.

We will use the following basic property about vertical connectivity,
which must be well known but we are not aware of a source to cite for
it.

\begin{lemma}\label{lem:decrcon}
  For any matroid $M$ with $|E(M)|\geq 2$ and any element $e$ of
  $E(M)$, we have $\kappa(M)-1\leq \kappa(M\del e)$.
\end{lemma}

\begin{proof}
  Let $k = \kappa(M\del e)$. By the discussion above, we can assume
  that $M$ has no loops.  The inequality is obvious if $M$ is
  disconnected, so assume that $M$ is connected.  Thus,
  $r(M)=r(M\del e)$.  If $k = r(M\del e)$, then $M\del e$ is not a
  union of two proper flats, which implies the same for $M$, so
  $\kappa(M)=r(M)$, and so the inequality holds.  Now assume that
  $k < r(M\del e)$.  Let $(X,Y)$ be a vertical $k$-separation of
  $M\del e$.

  If either $\lambda_{M\del e}(X) <k-1$ or $r(X\cup e)=r(X)$, then
  $(X\cup e,Y)$ is a vertical $k$-separation of $M$.  If
  $r(Y\cup e)=r(Y)$, then $(X,Y\cup e)$ is a vertical $k$-separation
  of $M$.  Each case yields $\kappa(M)\leq k$ and so gives the desired
  inequality.  Thus, we may assume that $\lambda_{M\del e}(X) =k-1$,
  that $r(X\cup e)>r(X)$, and that $r(Y\cup e)>r(Y)$.

  If $r(X)>k$, then $r(X)\geq k+1$, $r(Y\cup e)\geq k+1$, and
  $\lambda_M(X)=k$, so $(X,Y\cup e)$ is a $(k+1)$-separation of $M$,
  so $\kappa(M)-1\leq k$.  Thus, we may assume that $r(X)=k$.  The
  same argument, switching $X$ and $Y$, shows that we may assume that
  $r(Y)=k$.

  From $\lambda_{M\del e}(X) =k-1$ and $r(X)=k=r(Y)$ we get
  $k=r(M\del e)-1=r(M)-1$.  Since $\kappa(M)\leq r(M)$, the inequality
  $\kappa(M)-1\leq k$ follows.
\end{proof}

\subsection{Branch-width and tangles}\label{sec:bkbwt}

A tree $T$ is \emph{cubic} if each vertex that is not a leaf has
degree $3$.  Let $L(T)$ be the set of leaves of $T$.  A
\emph{branch-decomposition} $T_\phi$ of a matroid $M$ is a cubic tree
$T$ along with an injection $\phi:E(M)\to L(T)$; we call $a\in E(M)$
the \emph{label} on the leaf $\phi(a)$.  Given a branch-decomposition
of $M$, for each edge $e$ in $T$, the deletion $T-e$ has two
components, and the sets of labels on the leaves in these components
give complementary subsets $X$ and $\overline{X}$ of $E(M)$; we say
that the edge $e$ \emph{displays} $(X,\overline{X})$.  The \emph{width
  $w(e)$ of the edge $e$} that displays $(X,\overline{X})$ is
$\lambda_M(X)+1$.  The \emph{width $w(T_\phi)$ of the
  branch-decomposition $T_\phi$} is the maximum width of an edge in
$T$.  The \emph{branch-width} of $M$, denoted $bw(M)$, is the minimum
width $w(T_\phi)$ over all branch-decompositions of $M$.

Since $\lambda_M(X)=\lambda_{M^*}(X)$, it follows that
$bw(M)=bw(M^*)$.  Also, $w(e)=1$ if and only if the sets in the pair
that $e$ displays are separators of $M$, so $bw(M)=1$ if and only if
each element of $M$ is a loop or a coloop.

For tangles, we follow Geelen, Gerards, Robertson, and Whittle
\cite{tangle1}.  For a matroid $M$ and $k\in\mathbb{N}$, a
\emph{tangle of order} $k$ is a set $\mathcal{T}$ of subsets of $E(M)$
such that
\begin{itemize}
\item[(T1)] if $X\in \mathcal{T}$, then $\lambda_M(X)<k-1$,
\item[(T2)] if $X\subseteq E(M)$ and $\lambda_M(X)<k-1$, then either
  $X\in\mathcal{T}$ or $\overline{X}\in\mathcal{T}$,
\item[(T3)] if $X,Y,Z\in\mathcal{T}$, then $X\cup Y\cup Z\ne E(M)$,
  and
\item[(T4)] if $e\in E(M)$, then $E(M)-e\not\in\mathcal{T}$.
\end{itemize}

In \cite[Theorem 3.2]{tangle1}, Geelen, Gerards, Robertson, and
Whittle prove the following connection between branch-width and
tangles (in a broader setting than we state here).

\begin{thm}\label{thm:tangle}
  The maximum order of a tangle of a matroid $M$ is the branch-width,
  $bw(M)$.
\end{thm}

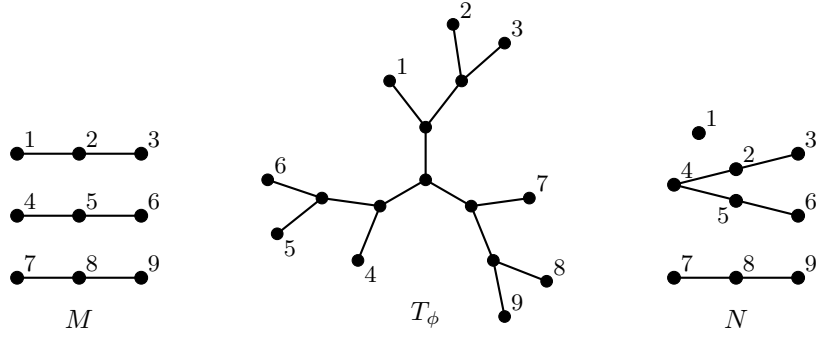
\begin{figure}
  \centering
  \begin{tikzpicture}[scale=1.1]
    \draw[thick](0,0)--(1.5,0);%
    \draw[thick](0,-.75)--(1.5,-.75);%
    \draw[thick](0,-1.5)--(1.5,-1.5);%
    \filldraw (0,0) node[above right =-1] {\small$1$} circle (2.2pt);%
    \filldraw (0.75,0) node[above right =-1] {\small$2$} circle
    (2.2pt);%
    \filldraw (1.5,0) node[above right =-1] {\small$3$} circle
    (2.2pt);%
    \filldraw (0,-0.75) node[above right =-1] {\small$4$} circle
    (2.2pt);%
    \filldraw (0.75,-0.75) node[above right =-1] {\small$5$} circle
    (2.2pt);%
    \filldraw (1.5,-0.75) node[above right =-1] {\small$6$} circle
    (2.2pt);%
    \filldraw (0,-1.5) node[above right =-1] {\small$7$} circle
    (2.2pt);%
    \filldraw (0.75,-1.5) node[above right =-1] {\small$8$} circle
    (2.2pt);%
    \filldraw (1.5,-1.5) node[above right =-1] {\small$9$} circle
    (2.2pt);%

    \node at (0.75,-2) {$M$};%
  \end{tikzpicture}
  \hspace{30pt}
  \begin{tikzpicture}[scale=1]
    \draw[thick] (0,0)--(90:0.7)--(70:1.4)--(60:2.1);%
    \draw[thick] (90:0.7)--(110:1.4);%
    \draw[thick] (70:1.4)--(80:2.1);%
    
    \draw[thick] (0,0)--(210:0.7)--(190:1.4)--(180:2.1);%
    \draw[thick] (210:0.7)--(230:1.4);%
    \draw[thick] (190:1.4)--(200:2.1);%

    \draw[thick] (0,0)--(330:0.7)--(310:1.4)--(300:2.1);%
    \draw[thick] (330:0.7)--(350:1.4);%
    \draw[thick] (310:1.4)--(320:2.1);%

    \filldraw (0,0) circle (2.2pt);%
    \filldraw (90:0.7) circle (2.2pt);%
    \filldraw (210:0.7) circle (2.2pt);%
    \filldraw (330:0.7) circle (2.2pt);%
    \filldraw (110:1.4) node[above right =-1] {\small$1$} circle
    (2.2pt);%
    \filldraw (70:1.4) circle (2.2pt);%
    \filldraw (80:2.1) node[above right =-1] {\small$2$} circle
    (2.2pt);%
    \filldraw (60:2.1) node[above right =-1] {\small$3$} circle
    (2.2pt);%

    \filldraw (230:1.4) node[below right =-1] {\small$4$} circle
    (2.2pt);%
    \filldraw (190:1.4) circle (2.2pt);%
    \filldraw (200:2.1) node[below right =-1] {\small$5$} circle
    (2.2pt);%
    \filldraw (180:2.1) node[above right =-1] {\small$6$} circle
    (2.2pt);%

    \filldraw (350:1.4) node[above right =-1] {\small$7$} circle
    (2.2pt);%
    \filldraw (310:1.4) circle (2.2pt);%
    \filldraw (320:2.1) node[above right =-1] {\small$8$} circle
    (2.2pt);%
    \filldraw (300:2.1) node[above right =-1] {\small$9$} circle
    (2.2pt);%

    \node at (0,-1.8) {$T_\phi$};%
  \end{tikzpicture}
  \hspace{30pt}
  \begin{tikzpicture}[scale=1.1]
    \draw[thick](0,-0.375)--(1.5,0);%
    \draw[thick](0,-0.375)--(1.5,-.75);%
    \draw[thick](0,-1.5)--(1.5,-1.5);%
    \filldraw (0.3,0.25) node[above right =-1] {\small$1$} circle (2.2pt);%
    \filldraw (0.75,-0.187) node[above right =-1] {\small$2$} circle
    (2.2pt);%
    \filldraw (1.5,0) node[above right =-1] {\small$3$} circle
    (2.2pt);%
    \filldraw (0,-0.375) node[above right =-1] {\small$4$} circle
    (2.2pt);%
    \filldraw (0.75,-0.57) node[below left =-1] {\small$5$} circle
    (2.2pt);%
    \filldraw (1.5,-0.75) node[above right =-1] {\small$6$} circle
    (2.2pt);%
    \filldraw (0,-1.5) node[above right =-1] {\small$7$} circle
    (2.2pt);%
    \filldraw (0.75,-1.5) node[above right =-1] {\small$8$} circle
    (2.2pt);%
    \filldraw (1.5,-1.5) node[above right =-1] {\small$9$} circle
    (2.2pt);%

    \node at (0.75,-2) {$N$};%
  \end{tikzpicture}
  \caption{A branch-decomposition of two rank-$3$ matroids, $M$ and
    $N$.  The width of each edge in $T_\phi$ that is incident with a
    leaf is $2$.  For $M$, the width of each other edge is $3$, so
    $w_M(T_\phi)=3$.  For $N$, the width of the vertical edge in the
    center is $4$, while all others that are not incident with leaves
    have width $3$, so $w_N(T_\phi)=4$.}\label{fig:bwexample}
\end{figure}

For example, for the matroid $M$ in Figure \ref{fig:bwexample}, a
branch-decomposition of width $3$ is shown, so $bw(M)\leq 3$.  The set
$\mathcal{T}$ of subsets of $E(M)$ of size at most one satisfies the
properties of a tangle of order $3$, so $bw(M)\geq 3$.  Thus,
$bw(M)=3$.  Likewise the branch-decomposition of the matroid $N$ in
that figure shows that $bw(N)\leq 4$, and the tangle of order $4$ that
consists of all sets of rank at most $2$ gives $bw(N)\geq 4$, so
$bw(N)=4$.  Note that $M$ and $N$ have the same configuration.  These
examples also illustrate the following result \cite[Lemma
3.2]{tangle}.

\begin{lemma}\label{lem:ggw}
  If $k\geq 3$, then each tangle of order $k$ in a matroid $M$
  contains all subsets $X$ of $E(M)$ for which $r(X)<k-1$.
\end{lemma}

Since $\lambda_M(X)\leq r(M)$ for all $X\subseteq E(M)$, we have that
$bw(M)\leq r(M)+1$.  The next result \cite[Proposition 2.2]{hall},
which characterizes the matroids for which that inequality is strict,
follows easily from Theorem \ref{thm:tangle} and Lemma \ref{lem:ggw}.

\begin{lemma}\label{lem:hall}
  For a matroid $M$, we have $bw(M)\leq r(M)$ if and only if $E(M)$ is
  the union of three proper flats of $M$.
\end{lemma}

\subsection{Cyclic sets, cyclic flats,  and clones}

A cyclic set of a matroid $M$ is a (possibly empty) union of circuits
of $M$; equivalently, $X\subseteq E(M)$ is cyclic if $M|X$ has no
coloops.  A cyclic flat is a flat that is cyclic.  The cyclic flats of
$M$ form a lattice, which we denote by $\mathcal{Z}(M)$.  The join of
$A$ and $B$ in $\mathcal{Z}(M)$ is $\cl_M(A\cup B)$, as in the lattice
of flats; the meet of $A$ and $B$ in $\mathcal{Z}(M)$ is $(A\cap B)-C$
where $C$ is the set of coloops of $M|A\cap B$.  We have
\begin{equation}\label{eq:rkviacyc}
  r(X) = \min\{r(A)+|X-A|\,:\,A\in\mathcal{Z}(M)\}
\end{equation}
for all subsets $X$ of $E(M)$.  Indeed,
$r(X) \leq r(X\cap A)+|X-A| \leq r(A)+|X-A|$ for all subsets $X$ and
$A$ of $E(M)$, and among the sets $A$ for which equality holds are the
closure of the union of the circuits of $M|X$, which is a cyclic flat,
and the cyclic flat $\cl(X)-C$ where $C$ is the set of coloops of
$M|\cl(X)$.

Thus, a matroid is determined by its cyclic flats along with the rank
of each cyclic flat.  The following characterization of matroids from
this perspective, which we will use when we define the $t$-expansion
of a matroid, is from \cite{sims, cycflats}.

\begin{thm}\label{thm:axioms}
  For a collection $\mathcal{Z}$ of subsets of a set $E$ and a
  function $r:\mathcal{Z}\to \mathbb{Z}$, there is a matroid $M$ on
  $E$ with $\mathcal{Z}(M)=\mathcal{Z}$ and $r_M(X) =r(X)$ for all
  $X\in\mathcal{Z}$ if and only if
  \begin{itemize}
  \item[(Z0)] $(\mathcal{Z},\subseteq)$ is a lattice,
  \item[(Z1)] $r(0_{\mathcal{Z}})=0$, where $0_{\mathcal{Z}}$ is the
    least set in $\mathcal{Z}$,
  \item[(Z2)] $0<r(Y)-r(X)<|Y-X|$ for all sets $X,Y$ in $\mathcal{Z}$
    with $X\subset Y$, and
  \item[(Z3)] for all pairs of sets $X,Y$ in $\mathcal{Z}$ (or,
    equivalently, just incomparable sets in $\mathcal{Z}$),
    $$r(X\join Y) + r(X\meet Y) + |(X\cap Y) - (X\meet Y)|\leq
    r(X)+r(Y).$$
  \end{itemize}
\end{thm}

Adaptations of the argument used to justify Equation
(\ref{eq:rkviacyc}) give the following lemma.

\begin{lemma}\label{lem:gencycrk}
  Let $M$ be a matroid. For any subset $X$ of $E(M)$,
  \begin{itemize}
  \item[(1)]
    $r(X) = \min\{r(A)+|X-A|\,:\,A\subseteq X \text{ and } A \text{ is
      a cyclic set}\}$, and
  \item[(2)] if
    $\mathcal{Z}(M)\subseteq \mathcal{Y}\subseteq 2^{E(M)}$, then
    $r(X) = \min\{r(A)+|X-A|\,:\,A\in\mathcal{Y}\}$.
  \end{itemize}
\end{lemma}

The cyclic sets of a matroid $M$ are the unions of circuits of $M$;
the flats of $M$ are the intersections of hyperplanes of $M$, so their
complements are the unions of cocircuits of $M$.  This gives the next
lemma, relating the cyclic flats of $M$ and those of its dual, $M^*$.

\begin{lemma}\label{lem:dualcy}
  For any matroid $M$, we have
  $\mathcal{Z}(M^*) = \{E(M)-A\,:\,A\in\mathcal{Z}(M)\}$.
\end{lemma}

A \emph{connected flat} of a matroid $M$ is a flat $F$ of $M$ for
which the restriction $M|F$ is connected.  All connected flats $F$
with $|F|\geq 2$ are cyclic flats, but the converse is false.

Distinct elements $e$ and $f$ in $M$ are \emph{clones} if the
transposition $(e,f)$ on $E(M)$, which maps $e$ to $f$, and $f$ to
$e$, and fixes each element of $E(M)-\{e,f\}$, is an automorphism of
$M$.  It is easy to show that $e$ and $f$ are clones if and only if
they are in the same cyclic flats of $M$.  With that reformulation of
clones, it follows that the relation on $E(M)$ given by $e\sim f$ if
and only if either $e=f$ or $e$ and $f$ are clones in $M$ is an
equivalence relation.  The equivalence classes of this relation are
the \emph{clonal classes} of $M$.  A \emph{set of clones} of $M$ is a
subset of a clonal class of $M$.  It is well known and easy to prove
that if $e$ and $f$ are clones in $M$, then they are clones in each
minor of $M$ that they are in.

\subsection{Positroids}

A \emph{positroid} is a matroid $M$, say of rank $r$ and with
$|E(M)|=n$, for which some matrix representation of $M$ by an $r$ by
$n$ matrix over $\mathbb{R}$ has the property that each submatrix that
is obtained by deleting all but $r$ columns has nonnegative
determinant.  Implicit in a matrix representation of a matroid $M$ is
a linear order on $E(M)$ since the columns of a matrix are naturally
linearly ordered.  A linear order on $E(M)$ in which the $i$th element
corresponds to the $i$th column of a matrix of the type above is a
\emph{positroid order} for $M$.  The next theorem, proven in
\cite{pos}, characterizes positroid orders and hence positroids.  It
uses the following notion: a subset $A$ of $E$ is a \emph{cyclic
  interval} in a linear order on $E$ if either $A$ or $E-A$ is an
interval.

\begin{thm}\label{thm:char}
  Let $M$ be a matroid that has no loops.  A linear order on $E(M)$ is
  a positroid order if and only if, for each proper connected flat $F$
  of $M$ with $|F|\geq 2$ and each connected component $K$ of the
  contraction $M/F$ with $|E(K)|\geq 2$, the set $E(K)$ is a subset of
  a cyclic interval that is disjoint from $F$.
\end{thm}

The property in Theorem \ref{thm:char} is the \emph{cyclic interval
  property}.

\subsection{Matroid unions}

The theory of submodular functions (see \cite[Chapter 11]{oxley})
yields the following well-known results about matroid union.  Let
$M_1, M_2, \ldots, M_k$ be matroids, all on the same set $E$.  Their
\emph{matroid union}, denoted $M_1\vee M_2\vee \cdots\vee M_k$, is
also defined on $E$ and is given by the rank function $r$ where
\begin{equation}\label{eq:rankunion}
r(X) = \min\biggl\{\biggl(\sum_{i=1}^kr_{M_i}(Y)\biggr)
+|X-Y|\,:\,Y\subseteq X\biggr\}
\end{equation}
for all $X\subseteq E$.  The independent sets of
$M_1\vee M_2\vee \cdots\vee M_k$ are the sets of the form
$I_1\cup I_2\cup\cdots\cup I_k$ where $I_j$ is independent in $M_j$.

It is well known that a matroid is transversal if and only if it is a
union of matroids of rank $1$.

\section{The $t$-expansion of a matroid}\label{section:texpansion}

From Theorem \ref{thm:axioms}, it is easy to check that the definition
below indeed defines a matroid.  The resulting matroid mimics the
structure of $M$, but with the size and rank of each cyclic flat
magnified by a factor of $t$.

\begin{dfn}\label{def:kexp}
  Fix a matroid $M$ and $t\in\mathbb{N}$.  For each $e\in E(M)$, let
  $S_e$ be a $t$-element set with $e\in S_e$ and
  $S_e\cap S_f=\emptyset$ whenever $e\ne f$.  For
  $X=\{e_1,e_2,\ldots,e_k\}\subseteq E(M)$, let
  $S_X=S_{e_1}\cup S_{e_2}\cup \cdots\cup S_{e_k}$.  The
  \emph{$t$-expansion} of $M$ is the matroid $M^t$ on $S_{E(M)}$ for
  which $\mathcal{Z}(M^t)=\{S_A \,:\, A\in \mathcal{Z}(M)\}$ and
  $r_{M^t}(S_A)=t\cdot r_M(A)$ for $S_A\in\mathcal{Z}(M^t)$.
\end{dfn}

The notation $S_e$, $S_X$, and $M^t$ is used throughout the rest of
this paper.

Note that any two distinct elements of any set $S_e$ are clones of
$M^t$.  Thus, each clonal class of $M^t$ is $S_X$ for some
$X\subseteq E(M)$.  Moreover, $S_X$ is a clonal class of $M^t$ if and
only if $X$ is a clonal class of $M$.  The following result is also
immediate from the definition.

\begin{lemma}\label{lem:expandconfig}
  Fix $t\in\mathbb{N}$.  Matroids $M$ and $N$, neither of which have
  coloops, have the same configuration if and only if $M^t$ and $N^t$
  have the same configuration.
\end{lemma}

\begin{lemma}\label{lem:kexpiso}
  Fix $t\in\mathbb{N}$.  The matroid $M$ is determined, up to
  isomorphism, by $M^t$.  Thus, $M^t$ and $N^t$ are isomorphic if and
  only if $M$ and $N$ are isomorphic.
\end{lemma}

\begin{proof}
  The second assertion follows from the first, so we focus on that.
  As noted above, each clonal class of $M^t$ is $S_X$ for some
  $X\subseteq E(M)$, and $S_X$ is a clonal class of $M^t$ if and only
  if $X$ is a clonal class of $M$.  From $M^t$, we know its clonal
  classes, and each has size $mt$ for some $m\in \mathbb{N}$.  From
  each clonal class $C$, say of size $mt$, pick any $m$ elements, and
  let $E_1$ be the set of all $|E(M)|$ such elements chosen.  Let
  $\mathcal{Z}_1 = \{E_1\cap A\,:\,A\in\mathcal{Z}(M^k)\}$ and let
  $r_1(E_1\cap A) = \frac{1}{t}\cdot r_{M^t}(A)$.  It is routine to
  check the conditions in Theorem \ref{thm:axioms} for the pair
  $(\mathcal{Z}_1,r_1)$, so this pair yields a matroid $M_1$ on $E_1$.
  For any two such matroids $M_1$ on $E_1$ and $M_2$ on $E_2$, there
  is a bijection $\phi:E_1\to E_2$ for which, for each $e\in E_1$, the
  elements $e$ and $\phi(e)$ are either equal or are clones in $M^t$.
  It follows that $\phi$ is an isomorphism.  Since one choice of $E_1$
  is $E(M)$, and in that case $r_1(A)=r_M(A)$ for all
  $A\in \mathcal{Z}(M)$, all matroids $M_1$ are isomorphic to $M$,
  which proves the result.
\end{proof}

The relation between $r_M(X)$ and $r_{M^t}(S_X)$ when
$X\in \mathcal{Z}(M)$ applies for all subsets $X$ of $E(M)$, as we
show next.

\begin{lemma}\label{lem:rank}
  If $X\subseteq E(M)$, then $r_{M^t}(S_X)=t\cdot r(X)$.
\end{lemma}

\begin{proof}
  By Equation (\ref{eq:rkviacyc}) and Definition \ref{def:kexp}, we
  have
  \begin{align*}
    r_{M^t}(S_X)=
    &\, \min\{r(S_Z)+|S_X- S_Z|:Z\in\mathcal{Z}(M)\}\\
    =&\,\min\{t\cdot  r(Z)+t|X- Z|:Z\in\mathcal{Z}(M) \}\\
    =&\, t\cdot \min\{r(Z)+|X- Z|:
       Z\in\mathcal{Z}(M)\}\\
    =&\, t\cdot r(X).\qedhere
  \end{align*}
\end{proof}

Routine arguments show that the operation of $t$-expansion commutes
with taking direct sums and free products.  (For the definition of
free products and the relevant background, see
\cite{free1,free2,oxley}.)  As we show next, the same is true for
taking duals.

\begin{lemma}\label{lem:dual}
  For any matroid $M$, we have $(M^*)^t=(M^t)^*$.  Thus, if $M$ is
  self-dual, so is $M^t$.
\end{lemma}

\begin{proof}
  By Definition \ref{def:kexp} and Lemma \ref{lem:dualcy},
  \begin{align*}
    \mathcal{Z}((M^*)^t)
    & =\{S_X\,:\,X\in\mathcal{Z}(M^*)\} \\
    &=\{S_{E(M)-Y}\,:\,Y\in\mathcal{Z}(M)\} \\
    &=\{S_{E(M)}-S_Y\,:\,S_Y\in\mathcal{Z}(M^t)\} \\
    & = \mathcal{Z}((M^t)^*).
  \end{align*}
  The result now follows since the rank of any set $S_X$ is the same
  in $(M^*)^t$ as in $(M^t)^*$:
  \begin{align*}
    r_{(M^*)^t}(S_X)
    & = t\cdot r_{M^*}(X) \\
    &= t\cdot (|X|  +r_M(E(M)-X)-r(M)) \\
    & = |S_X|+r_{M^t}(S_{E(M)}-S_X) -r(M^t)\\
    & = r_{(M^t)^*}(S_X).\qedhere
  \end{align*}
\end{proof}

Taking $t$-expansions commutes with certain minors, as we show next.

\begin{lemma}\label{lem:texpminors}
  If $X\subseteq E(M)$, then $M^t|S_X=(M|X)^t$ and
  $M^t/S_X=(M/X)^t$.
\end{lemma}

\begin{proof}
  Duality relates deletion and contraction, so by Lemma
  \ref{lem:dual}, it suffices to prove that $M^t|S_X=(M|X)^t$.  All
  elements of any set $S_e$ are clones in $M^t$, so they are clones in
  any restriction that contains them, and so any cyclic flat of
  $M^t|S_X$ is $S_A$ for some $A\subseteq X$.  The lemma now follows
  from Lemma \ref{lem:gencycrk} since
  $r_{M^t|S_X}(S_A)=t\cdot r_M(A)=r_{(M|X)^t}(S_A)$.
\end{proof}

The next four lemmas relate the flats of $M$ and those of $M^t$.
  
\begin{lemma}\label{lem:flatstructure1}
  For $A\subseteq E(M)$ and $e\in A$, the following statements are
  equivalent:
  \begin{enumerate}
  \item $e$ is a coloop of $M|A$,
  \item all elements of $S_e$ are coloops of $M^t|S_A$,
  \item some element of $S_e$ is a coloop of $M^t|S_A$.
  \end{enumerate}
  Thus, $X\subseteq E(M)$ is cyclic in $M$ if and only if $S_X$ is
  cyclic in $M^t$.  Also, $F\subseteq E(M)$ is a flat of $M$ if and
  only if $S_F$ is a flat of $M^t$.
\end{lemma}
  
\begin{proof}
  Items (2) and (3) are equivalent since $S_e$ is a set of clones of
  $M^t$.  The equivalence of items (1) and (2) follows from Lemma
  \ref{lem:rank}.  The statement about $F$ and $S_F$ follows since $F$
  is a flat of $M$ if and only if, for each $x\in E(M)-F$, the element
  $x$ is a coloop of $M|F\cup x$, and likewise for flats of $M^t$.
\end{proof}

From Lemma \ref{lem:rank}, a set $S_X$ is a separator of $M^t$ if and
only if $X$ is a separator of $M$.  Since each set $S_e$, for
$e\in E(M)$, is a set of clones, if $M$ has no coloops (so $M^t$ has
none), then each separator of $M^t$ is $S_X$ for some
$X\subseteq E(M)$.  These observations along with Lemmas
\ref{lem:texpminors} and \ref{lem:flatstructure1} yield the next
lemma.

\begin{lemma}\label{lem:conn}
  A matroid $M$ is connected if and only if $M^t$ is connected.  Also,
  the connected flats of $M^t$ with at least two elements are the sets
  $S_F$ where $F$ is a connected flat of $M$ with $|F|\geq 2$.
\end{lemma}

\begin{lemma}\label{lem:flatstructure2}
  For $X\subseteq E(M^t)$, let
  $\theta(X) = \{e\in E(M)\,:\,S_e\subseteq X\}$.  If $X$ is a flat of
  $M^t$, then $\theta(X)$ is a flat of $M$ and each element of
  $X-S_{\theta(X)}$ is a coloop of $M^t|X$.
\end{lemma}

\begin{proof}
  We first show that $S_{\theta(X)}$ is a flat of $M^t$.  If
  $y\in \cl_{M^t}(S_{\theta(X)})-S_{\theta(X)}$, then there would be a
  circuit $C$ of $M^t$ with $y\in C$ and $C-y\subseteq S_{\theta(X)}$.
  Now $y\in S_e$ for some $e\in E(M)$.  Since $S_e$ is a set of
  clones, $(C-y)\cup x$ would be a circuit of $M^t$ for any
  $x\in S_e$.  Since $X$ is a flat and
  $C-y\subseteq S_{\theta(X)}\subseteq X$, this would give
  $S_e\subseteq X$, so $e\in \theta(X)$, but that is impossible since
  $y\in S_e- S_{\theta(X)}$.  Thus, $S_{\theta(X)}$ is a flat of
  $M^t$, so $\theta(X)$ is a flat of $M$ by Lemma
  \ref{lem:flatstructure1}.  The same argument about circuits and
  clones shows that if $z\in X-S_{\theta(X)}$, then no circuit of
  $M^t|X$ contains $z$, so $z$ is a coloop of $M^t|X$.
\end{proof}

For a flat $X$ of $M^t$, some elements of $S_{\theta(X)}$ may also be
coloops of $M^t|X$.

\begin{lemma}\label{lem:flatstructure3}
  Let $F$ be a flat of $M$.  If $A\subset E(M^t)-S_F$ with $|A|<t$,
  then $S_F\cup A$ is a flat of $M^t$.
\end{lemma}

\begin{proof}
  Assume, to the contrary, that
  $x\in\cl_{M^t}(S_F\cup A)- \bigl(S_F\cup A\bigr)$.  Say $x\in S_e$.
  Since $x$ is not a coloop of $M^t|\cl_{M^t}(S_F\cup A)$, Lemma
  \ref{lem:flatstructure2} gives $S_e\subseteq \cl_{M^t}(S_F\cup A)$.
  Now $r_M(F\cup e)=r_M(F)+1$, so
  $$r_{M^t}(S_F\cup A)=r_{M^t}(\cl_{M^t}(S_F\cup A))
  \geq r_{M_t}(S_F\cup S_e)=r_{M^t}(S_F)+t,$$ but that contradicts the
  inequality
  $$r_{M^t}(S_F\cup A)\leq r_{M_t}(S_F)+|A|<r_{M^t}(S_F)+t.$$  Thus,
  $\cl_{M^t}(S_F\cup A)= S_F\cup A$, so $S_F\cup A$ is a flat of
  $M^t$.
\end{proof}

As we show next, the $t$-expansion of a matroid can be written as a
matroid union, and if $M$ itself is a matroid union, there can be many
such expressions.

\begin{thm}\label{thm:texpviaunion}
  Fix $t\in\mathbb{N}$.  Let $M$ be the matroid union
  $M_1\vee M_2\vee \cdots \vee M_k$ where $k\geq 1$.  For each
  $i\in [k]$ and $j\in[t]$, let $M_{i,j}$, be the matroid obtained
  from $M_i$ by, for each $e\in E(M)$, adding each element of $S_e-e$
  parallel to $e$ if $r_{M_i}(e)=1$, or as a loop if $r_{M_i}(e)=0$.
  Then $M^t$ is the matroid union
  \begin{equation}\label{eq:bigunion}
    \bigvee_{i\in [k],j\in[t]} M_{i,j}.
  \end{equation}
\end{thm}

\begin{proof}
  Let $M'$ be the matroid union (\ref{eq:bigunion}) and let $r'$ be
  its rank function.  By construction, each set $S_e$, for
  $e\in E(M)$, is a set of clones in each matroid $M_{i,j}$, and so
  $S_e$ is a set of clones in $M'$.  Thus, each cyclic flat of $M'$ is
  $S_X$ for some $X\subseteq E(M)$.  Therefore, by Lemma
  \ref{lem:gencycrk}, in order to show that $M'=M^t$, it suffices to
  show that $r'(S_X)=r_{M^t}(S_X)$ for all $X\subseteq E(M)$.  By
  Equation (\ref{eq:rankunion}), for $X\subseteq E(M)$,
  $$r'(S_X)
  = \min\biggl\{\biggl(\sum_{i\in [k],j\in[t]}r_{M_{i,j}}(Y)\biggr) +
  |S_X-Y|\,:\,Y\subseteq S_X\biggr\}.$$ If $S_e\cap Y\ne \emptyset$,
  then the sum can only decrease or remain the same if we replace $Y$
  by $Y\cup S_e$.  Also, $r_{M_{i,j}}(S_V)=r_{M_i}(V)$ for
  $V\subseteq X$, so
  $$r'(S_X) = \min\biggl\{\biggl(\sum_{i\in [k]}t\cdot r_{M_i}(V)\biggr)
  + |S_X-S_V|\,:\,V\subseteq X\biggr\}.$$

  By Lemma \ref{lem:flatstructure1}, the union of the circuits of
  $M^t|S_X$ is $S_W$ for some $W\subseteq X$, so by Lemmas
  \ref{lem:gencycrk} and \ref{lem:rank},
  \begin{align*}
    r_{M^t}(S_X)
    & = \min\{r_{M^t}(S_W)+  |S_X-S_W|\,:\,W\subseteq X\}\\
    & =  \min\{t\cdot r_M(W) + t\cdot |X-W|\,:\,W\subseteq X\}.       
  \end{align*}
  By Equation (\ref{eq:rankunion}),
  $$r_M(W) = \min\biggl\{\biggl(\sum_{i\in [k]}r_{M_i}(U)\biggr) +
  |W-U|\,:\,U\subseteq W\biggr\}.$$ Now for
  $U\subseteq W\subseteq X$, we have
  $$t\cdot |W-U|+t\cdot|X-W| =t\cdot |X-U| =|S_X-S_U|.$$
  Therefore
  \begin{align*}
    r_{M^t}(S_X)
    & =  \min\{t\cdot r_M(W) + t\cdot |X-W|\,:\,W\subseteq X\}\\
    & =  \min\biggl\{\biggl(t\cdot
      \sum_{i\in [k]}r_{M_i}(U)\biggr)+t\cdot |W-U|  +
      t\cdot |X-W|\,:\,U\subseteq W\subseteq X\biggr\}\\
    & =  \min\biggl\{ \biggl( \sum_{i\in [k]}t\cdot r_{M_i}(U)\biggr) + |S_X-S_U|
      \,:\,U\subseteq X\biggr\}\\        
    & = r'(S_X). \qedhere
  \end{align*}
\end{proof}

\begin{cor}\label{cor:transv}
  If $M$ is a transversal matroid, so is $M^t$.  If $M$ is a
  cotransversal matroid (i.e., $M^*$ is transversal), so is $M^t$.
\end{cor}

\begin{proof}
  When $M$ is transversal, we can write it as
  $M_1\vee M_2\vee \cdots \vee M_k$ where each matroid $M_i$ has rank
  $1$. Thus, when applying the theorem above, each $M_{i,j}$ has rank
  $1$, and so $M^t$, being a matroid union of rank-$1$ matroids, is
  transversal.  The corresponding statement for cotransversal matroids
  follows by Lemma \ref{lem:dual}.
\end{proof}

Adding the elements of each set $S_e-e$ to $M$, as in the theorem, to
get $N$ and taking the matroid union of $t$ copies of $N$, while also
giving $M^t$, is quite different from what is done in the proof of the
corollary since $N$ need not be transversal; the class of transversal
matroids is not closed under parallel extension.  This illustrates the
fact that different ways to write $M$ as a matroid union can yield
significantly different ways to write $M^t$ as a matroid union.

When a transversal matroid $M$ is written as
$M_1\vee M_2\vee \cdots \vee M_k$ where each matroid $M_i$ has rank
$1$, if $A_i$ is the set of non-loops of $M_i$, then
$(A_1,A_2,\ldots,A_k)$ is a presentation of $M$; likewise, we can go
from presentations to matroid unions by reversing this.  It follows
that if $(A_1,A_2,\ldots,A_k)$ is a presentation of $M$ and, for each
$i\in[k]$ and $j\in[t]$, we set $A_{i,j}=S_{A_i}$, then
$(A_{1,1},A_{1,2},\ldots,A_{1,t},\ldots,A_{k,1},A_{k,2},\ldots,A_{k,t})$
is a presentation of $M^t$. One can adapt those ideas to show that the
class of lattice path matroids \cite{lpm1} and the class of multi-path
matroids \cite{multipath} are closed under $t$-expansion.

Another way to prove Corollary \ref{cor:transv} yields a stronger
result.  The Mason-Ingleton inequalities characterize transversal
matroids, and the corresponding equalities characterize fundamental
(or principal) transversal matroids (see \cite{chartrans}).  With that
approach, the result below follows easily.

\begin{thm}\label{thm:trcl}
  Fix $t\in\mathbb{N}$.  A matroid $M$ is transversal if and only if
  $M^t$ is transversal.  Furthermore, $M$ is fundamental transversal
  if and only if $M^t$ is fundamental transversal.
\end{thm}

Since the class of gammoids (minors of transversal matroids) is closed
under parallel extension and matroid unions, we have the following
corollary of Theorem \ref{thm:texpviaunion}.

\begin{cor}
  For any $t\in \mathbb{N}$, the $t$-expansion of a gammoid is a
  gammoid.
\end{cor}

We next prove a counterpart of Theorem \ref{thm:trcl} for positroids.

\begin{thm}\label{thm:poscl}
  Fix $t\in\mathbb{N}$.  A matroid $M$ is a positroid if and only if
  $M^t$ is a positroid.
\end{thm}

\begin{proof}
  We use the characterization of positroids in Theorem \ref{thm:char}.
  We write a linear order as a list, giving the elements from least to
  greatest.  First assume that $M$ is a positroid, and let
  $x,y,\ldots,z$ be a positroid order for $M$.  For each $e\in E(M)$,
  take a linear order $e,e_1,e_2,\ldots,e_{t-1}$ on $S_e$, and
  concatenate these linear orders as
  $$x,x_1,x_2,\ldots,x_{t-1}, y,y_1,y_2,\ldots,y_{t-1},\ldots,
  z,z_1,z_2,\ldots,z_{t-1}$$ to get a linear order on $E(M^t)$.  We
  claim that this is a positroid order for $M^t$.  By Lemma
  \ref{lem:conn}, a connected flat of $M^t$ with at least two elements
  is $S_F$ for some connected flat $F$ of $M$ with $|F|\geq 2$.  The
  lemmas above, applied to $M/F$ and $(M/F)^t=M^t/S_F$, show that the
  non-singleton connected components of $M^t/S_F$ are the
  $t$-expansions of the non-singleton connected components of $M/F$.
  Each non-singleton connected component of $M/F$ is a subset of a
  cyclic interval of $E(M)$ that is disjoint from $F$, so, using the
  linear order on $E(M^t)$ above, the same follows for the
  non-singleton connected components of $M^t/S_F$.  Thus, the cyclic
  interval property holds for $M^t$, so it is a positroid.  A similar
  idea yields the converse, using the linear order induced on $E(M)$
  by a linear order on $E(M^t)$ (in which the sets $S_e$ may or may
  not be intervals).
\end{proof}

\section{Tutte connectivity}\label{section:connectivity}

We first show that $\tau(M)$ and $\tau(M^t)$ are related in a simple
way when $\tau(M)<\infty$.  We then use that relation to show that,
for each $n\in \mathbb{N}$, there are loopless matroids $M'$ and $N'$
on the same ground set that have the same configuration (and so the
same Tutte polynomial and $\mathcal{G}$-invariant) but
$\tau(N')-\tau(M')\geq n$.

\begin{thm}\label{thm:connkexp}
  For a matroid $M$ and $t\in \mathbb{N}$, if $\tau(M)<\infty$, then
  $\tau(M^t) =t(\tau(M)-1)+1$.
\end{thm}

\begin{proof}
  Set $\tau(M)=k$ and $\tau(M^t)=k'$.  Let $(X,\overline{X})$ be a
  $k$-separation of $M$.  Thus,
  $$\lambda_{M^t}(S_X) = t\cdot\lambda_M(X)\leq t(k-1).$$
  Also, $|X|\geq k$ so $|S_X|\geq t\cdot k\geq t(k-1)+1$ and likewise
  for $|S_{\overline{X}}|$, so $(S_X, S_{\overline{X}})$ is a
  $(t(k-1)+1)$-separation of $M^t$.  Thus, $k'\leq t(k-1)+1$.  We will
  get $k'\geq t(k-1)+1$ by first proving that $M^t$ has a
  $k'$-separation of the form $(S_Y, S_{\overline{Y}})$.

  Let $(A,\overline{A})$ be a $k'$-separation of $M^t$.  Assume that,
  for some $e\in E(M)$, both $S_e\cap A$ and $S_e\cap \overline{A}$
  are nonempty.  Distinct elements of $S_e$ are clones in $M^t$, so
  the same holds in each minor of $M^t$ that they are in.  This gives
  three options:
  \begin{enumerate}
  \item all elements of $S_e\cap A$ are coloops of $M^t|A$ and no
    element of $S_e\cap \overline{A}$ is a coloop of
    $M^t|\overline{A}$ (or the reverse, switching $A$ and
    $\overline{A}$),
  \item each element of $S_e$ is a coloop of either $M^t|A$ or
    $M^t|\overline{A}$, or
  \item no element of $S_e$ is a coloop of either $M^t|A$ or
    $M^t|\overline{A}$.
  \end{enumerate}
  In case (1), set $B=A-S_e$. Then
  $\lambda_{M^t}(B)=\lambda_{M^t}(A)- |S_e\cap A|$.  Also,
  $B\ne\emptyset$, for otherwise case (1) would give
  $r_{M^t}(\overline{A})=r_{M^t}(\overline{B})=r(M^t)$ and
  $r_{M^t}(A)=|A|$, and so $\lambda_{M^t}(A)< k'$ would give
  $|A|< k'$, contrary to $(A,\overline{A})$ being a $k'$-separation.
  Thus, $(B,\overline{B})$ is a $(k'-|A\cap S_e|)$-separation,
  contrary to the choice of $k'$.  Thus, case (1) does not occur.

  In both cases (2) and (3), we get $\lambda_{M^t}(A-S_e)< k'$ and
  $\lambda_{M^t}(A\cup S_e)< k'$ from $\lambda_{M^t}(A)< k'$, so we
  can resolve each instance of case (2) or (3) in at least one of
  these ways (and so get the desired $k'$-separation
  $(S_Y, S_{\overline{Y}})$) provided that either $|A-S_e|\geq k'$ or
  $|\overline{A}-S_e|\geq k'$.  Assume instead that $|A-S_e|< k'$ and
  $|\overline{A}-S_e|< k'$.  Then
  $$t|E|=|S_E|=|A- S_e|+|\overline{A}- S_e|+|S_e|\leq 2k'-2+t.$$
  Since $k'\leq t(k-1)+1$, this would give $t|E|\leq 2(t(k-1)+1)-2+t$,
  which simplifies to $|E|\leq 2k-1$, but since $M$ has
  $k$-separations, $|E|\geq 2k$.

  Thus, there is a $k'$-separation $(S_Y, S_{\overline{Y}})$ of $M^t$.
  Now $t\cdot\lambda_M(Y)=\lambda_{M^t}(S_Y)= k'-1$ since
  $\tau(M^k)=k'$.  From
  $t\cdot|Y|=|S_Y|\geq k'=t\cdot \lambda_M(Y)+1$, we get
  $|Y|\geq \lambda_M(Y)+1$; likewise,
  $|\overline{Y}|\geq \lambda_M(Y)+1$.  So $(Y,\overline{Y})$ is an
  $(\lambda_M(Y)+1)$-separation of $M$.  Thus, $k\leq \lambda_M(Y)+1$,
  so $t(k-1)+1 \leq t\cdot \lambda_M(Y)+1 =k'$. With the inequality
  proven in the first paragraph, this gives $t(k-1)+1 =k'$.
\end{proof}

Recall that $\tau(M)=\infty$ if and only if $M$ is the uniform matroid
$U_{r,n}$ with $r\in\mathbb{N}$ and $n\in\{2r-1,2r,2r+1\}$.  The
$t$-expansion $U_{tr,2tr}$ of $U_{r,2r}$ has the same form, so
$\tau((U_{r,2r})^t)=\infty= \tau(U_{r,2r})$.  In contrast, since
$(U_{r,2r-1})^t= U_{tr,2tr-t}$, if $t>1$, then
$\tau(U_{r,2r-1})=\infty$ while $\tau((U_{r,2r-1})^t)\ne\infty$.
Likewise, for vertical connectivity we will see that if
$\kappa(M)=r(M)$, then $\kappa(M^t)$ might, but does not have to, be
$r(M^t)$.

\begin{thm}\label{thm:Tuttegap}
  For any $n\in\mathbb{N}$, there are matroids $M'$ and $N'$ with no
  coloops and with the following properties: each is self-dual, both
  are positroids, both are transversal matroids, they have the same
  configuration, and $\tau(N')-\tau(M')\geq n$.
\end{thm}

\begin{proof}
  Let $M$ and $N$ be the matroids shown in Figure
  \ref{fig:configexamples}.  Both are positroids, transversal, and
  self-dual.  Now $\tau(N)=3$ and $\tau(M)=2$, so
  $\tau(N^n)-\tau(M^n)=2n+1-(n+1)=n$ for any $n\in\mathbb{N}$.  The
  remaining assertions hold by Lemma \ref{lem:dual}, Theorem
  \ref{thm:poscl}, Corollary \ref{cor:transv}, and Lemma
  \ref{lem:expandconfig}.
\end{proof}

\section{Vertical connectivity}\label{section:vertical}

In this section, we show that if $\kappa(M)<r(M)$, then $\kappa(M)$
and $\kappa(M^t)$ are related in a manner like that for $\tau(M)$ and
$\tau(M^t)$ in Theorem \ref{thm:connkexp}.
We also give a counterpart
of Theorem \ref{thm:Tuttegap} for vertical connectivity.

We start with an example of the case not covered by Theorem
\ref{thm:vertkexp}.  The matroids $M$ and $N$ shown in Figure
\ref{fig:kappa=rank} both have vertical connectivity $3$ since $E(M)$
is not the union of two hyperplanes of $M$, and likewise for $N$.
However, $M^2$ has vertical $5$-separations, such as
$(X,\overline{X})$ with $X=S_1\cup S_2\cup S_3\cup \{6\}$, while
$\kappa(N^2)=6=r(N^2)$ since $E(N^2)$ is not a union of two
hyperplanes of $N^2$.  Thus, $\kappa(M)=\kappa(N)$, but
$\kappa(M^2)\ne\kappa(N^2)$.

\begin{figure}
  \centering
  \begin{tikzpicture}[scale=1]
    \filldraw (0,0.5) node[above] {\small$1$} circle (2pt);%
    \filldraw (1,0.75) node[above] {\small$2$} circle (2pt);%
    \filldraw (2,1) node[above] {\small$3$} circle (2pt);%
    \filldraw (1,0.25) node[above] {\small$4$} circle (2pt);%
    \filldraw (2,0) node[above] {\small$5$} circle (2pt);%
    \filldraw (2.5,0.5) node[above] {\small$6$} circle (2pt);%
    \draw[thick](0,0.5)--(2,0);%
    \draw[thick](0,0.5)--(2,1);%
    \node at (1,-0.75) {$M$};%
    \filldraw (4,0.5) node[above] {\small$1$} circle (2pt);%
    \filldraw (5,0.75) node[above] {\small$2$} circle (2pt);%
    \filldraw (6,1) node[above] {\small$3$} circle (2pt);%
    \filldraw (5,0.25) node[above] {\small$4$} circle (2pt);%
    \filldraw (6,0) node[above] {\small$5$} circle (2pt);%
    \filldraw (6.5,0.1) node[above] {\small$6$} circle (2pt);%
    \filldraw (6.5,0.9) node[above] {\small$7$} circle (2pt);%
    \draw[thick](4,0.5)--(6,0);%
    \draw[thick](4,0.5)--(6,1);%
    \node at (5,-0.75) {$N$};%
  \end{tikzpicture}
  \caption{ Two rank-$3$ matroids with $\kappa(M)=\kappa(N)$, but
    $\kappa(M^2)\ne\kappa(N^2)$.}
  \label{fig:kappa=rank}
\end{figure}
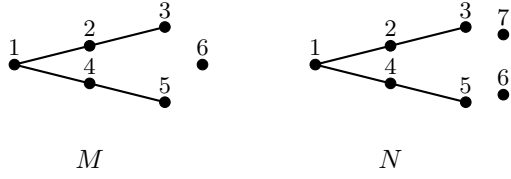

\begin{thm}\label{thm:vertkexp}
  For a matroid $M$, if $\kappa(M)<r(M)$, then
  $\kappa(M^t)= t(\kappa(M)-1)+1$ for each integer $t\in\mathbb{N}$.
\end{thm}

\begin{proof}
  Since $\kappa(M)=\kappa(M\del L)$ where $L$ is the set of loops of
  $M$, we may assume that $M$ has no loops.  If $M$ is disconnected,
  then so is $M^t$, and so, since $M$ has no loops, $\kappa(M)=1$ and
  $\kappa(M^t)=1= t(\kappa(M)-1)+1$.

  We induct on $|E(M)|$.  Since $\kappa(M)<r(M)$, the least that
  $|E(M)|$ can be is $2$, in which case $M$ is $U_{2,2}$, so the
  equality in the theorem holds by the remark on disconnected loopless
  matroids.  Now assume that $|E(M)|\geq 3$ and that the result holds
  for all matroids $N$ for which $\kappa(N)<r(N)$ and
  $|E(N)|=|E(M)|-1$.  We may also assume that $M$ is connected.
  
  Set $\kappa(M)=k$ and $\kappa(M^t)=k'$.  Let $(X,\overline{X})$ be a
  vertical $k$-separation of $M$.  Then
  $$\lambda_{M^t}(S_X) =t\cdot \lambda_M(X) \leq t(k-1).$$  Also,
  $r_{M^t}(S_X)=t\cdot r(X)\geq t\cdot k \geq t(k-1)+1$ and likewise for
  $r_{M^t}({\overline{S_X}})$, so $(S_X, \overline{S_X})$ is a
  vertical $(t(k-1)+1)$-separation of $M^t$.  Thus,
  $k'\leq t(k-1)+1<r(M^t)$.

  If $M^t$ has a vertical $k'$-separation of the form
  $(S_Y, \overline{S_Y})$, then we get $k'\geq t(k-1)+1$ by adapting
  the argument in the last paragraph of the proof of Theorem
  \ref{thm:connkexp}, from which the desired equality follows.  Thus,
  assume that no vertical $k'$-separation of $M^t$ has the form
  $(S_Y, \overline{S_Y})$.
  For a vertical $k'$-separation $(A,\overline{A})$ of $M^t$, set
  $$\sigma(A)=\{e\in E\,:\, S_e\cap A\ne\emptyset \text{ and }
  S_e\cap \overline{A}\ne\emptyset\}.$$ Let $(A,\overline{A})$ be a
  vertical $k'$-separation of $M^t$ with $|\sigma(A)|$ minimal.
    
  Assume that $e\in \sigma(A)$.  Distinct elements of $S_e$ are clones
  in $M^t$, so the same holds in any minor of $M^t$ that they are in.
  This gives three options:
  \begin{enumerate}
  \item all elements of $S_e\cap A$ are coloops of $M^t|A$ and no
    element of $S_e\cap \overline{A}$ is a coloop of
    $M^t|\overline{A}$ (or the reverse, switching $A$ and
    $\overline{A}$),
  \item no element of $S_e$ is a coloop of either $M^t|A$ or
    $M^t|\overline{A}$, or
  \item each element of $S_e$ is a coloop of either $M^t|A$ or
    $M^t|\overline{A}$.
  \end{enumerate}
  In case (1), set $B=A-S_e$.  Then
  $\lambda_{M^t}(B)=\lambda_{M^t}(A)- |S_e\cap A|$.  Also,
  $B\ne\emptyset$, for otherwise we would have
  $r_{M^t}(\overline{A})=r_{M^t}(\overline{B})=r(M^t)$, and so
  $\lambda_{M^t}(A)< k'$ would give $r_{M^t}(A)< k'$, contrary to
  $(A,\overline{A})$ being a vertical $k'$-separation.  Thus,
  $(B,\overline{B})$ is a vertical $(k'-|A\cap S_e|)$-separation of
  $M^t$, contrary to $k'$ being $\kappa(M^t)$.  So case (1) does not
  occur.

  In case (2), if $x\in S_e\cap A$, then $r_{M^t}(A-x)=r_{M^t}(A)$ and
  $r_{M^t}(\overline{A-x})=r_{M^t}(\overline{A})$, so
  $(A-x,\overline{A-x})$ is a vertical $k'$-separation of $M^t$.  If
  $S_e\cap (A-x)\ne\emptyset$, then, since no element of
  $S_e\cap \overline{A-x}$ is a coloop of $M^t|\overline{A-x}$ and
  case (1) does not occur, no element of $S_e\cap (A-x)$ is a coloop
  of $M^t|A-x$, so we can repeat the argument.  It follows that
  $(A-S_e,\overline{A-S_e})$ is a vertical $k'$-separation of $M^t$,
  contrary to $|\sigma(A)|$ being minimal.  Thus, case (2) does not
  occur.
  
  We now focus on case (3). We claim that $\sigma(A)= \{e\}$.  Assume
  that $f\in \sigma(A)-e$.  Let
  $$m=\min\{|A\cap S_e|, |\overline{A}\cap S_e|,|A\cap S_f|,
  |\overline{A}\cap S_f|\}.$$ By symmetry, we may assume that
  $|A\cap S_e|=m$.  Let $B=(A-S_e)\cup V$ where
  $V\subseteq \overline{A}\cap S_f$ and $|V|=m$.  The elements removed
  from $A$ are coloops of $M|A$, and likewise for $\overline{A}$, so
  $r_{M^t}(B)\leq r_{M^t}(A)$ and
  $r_{M^t}(\overline{B})\leq r_{M^t}(\overline{A})$.  Also, exchanging
  elements one at a time and recalling that case (1)
  does not occur shows that
  $r_{M^t}(A)-1\leq r_{M^t}(B)$ and
  $r_{M^t}(\overline{A})-1\leq r_{M^t}(\overline{B})$.
    It follows that $r_{M^t}(B)= r_{M^t}(A)$ and
  $r_{M^t}(\overline{B})= r_{M^t}(\overline{A})$ since $M^t$ has no
  vertical $(k'-1)$-separation.  Thus, $(B,\overline{B})$ is a
  vertical $k'$-separation of $M^t$; also, $|\sigma(B)|<|\sigma(A)|$,
  which contradicts the choice of $(A,\overline{A})$.  Thus,
  $\sigma(A)= \{e\}$.
  
  Consider $M^t\del S_e$, which is $(M\del e)^t$.  Since $M$ is
  connected, $r(M^t\del S_e) = r(M^t)$.  We have
  \begin{align*}
    r_{M^t}(A-S_e)+r_{M^t}(\overline{A}-S_e)
    & =r_{M^t}(A)+r_{M^t}(\overline{A})-t\\
    & <r(M^t)+k'-t\\
    & =r(M^t\del S_e)+k'-t.
  \end{align*}
  Also, $r(A-S_e)>k'-t$ and $r(\overline{A}-S_e)> k'-t$, so
  $(A-S_e,\overline{A}-S_e)$ is a vertical $(k'-t)$-separation of
  $M^t\del S_e$.  Applying Lemma \ref{lem:decrcon} to $M^t$ a total of
  $t$ times shows that $k'-t\leq\kappa(M^t\del S_e)$, so having a
  vertical $(k'-t)$-separation of $M^t\del S_e$ gives the equality
  $k'-t=\kappa(M^t\del S_e)$.  Since $M$ is connected,
  $r(M\del e)=r(M)$, so $ \kappa(M\del e)<r(M\del e)$.  Thus, the
  induction hypothesis applies to $M\del e$ and gives
  $\kappa(M^t\del S_e) =t(\kappa(M\del e) -1)+1$.  That equality and
  Lemma \ref{lem:decrcon} applied to $M$ give
  $\kappa(M^t\del S_e) \geq t(k -2)+1$, so
  $$k'=t+\kappa(M^t\del S_e) 
  \geq t(k-1)+1, $$ which is the inequality we needed to show in order
  to get $k'= t(k-1)+1$.
\end{proof}

\begin{thm}\label{thm:vertgap}
  For any $n\in\mathbb{N}$, there are matroids $M'$ and $N'$ with no
  coloops and with the following properties: each is self-dual, both
  are positroids, both are transversal matroids, they have the same
  configuration, and $\kappa(N')-\kappa(M')\geq n$.
\end{thm}
 
\begin{proof}
  Let $M_0$ and $N_0$ be the $2$-expansions of the matroids $M$ and
  $N$, respectively, shown in Figure \ref{fig:configexamples}.  Now
  $\kappa(M_0)=2(2-1)+1=3$ by Theorem \ref{thm:vertkexp}, but that
  result does not apply to $N_0$.  Note that $(X,\overline{X})$, where
  $X=S_1\cup S_2\cup S_3\cup\{6\}$, is a vertical $5$-separation of
  $N_0$.  One can check that if $r_{N_0}(Y)\leq 4$, then
  $r_{N_0}(\overline{Y})=6$, from which it follows that $N_0$ has no
  vertical $j$-separation with $j<5$, so $\kappa(N_0)=5$.  Hence, for
  any $t\in\mathbb{N}$, we have, for the $t$-expansions,
  $\kappa(N_0^t)-\kappa(M_0^t)=4t+1-(2t+1)=2t$. The other properties
  hold as in the proof of Theorem \ref{thm:Tuttegap}, so choosing $t$
  with $t\geq n/2$ proves the result.
\end{proof}

Theorem \ref{thm:vertkexp} applies when $\kappa(M)<r(M)$.  To close
this section, we identify when that condition applies to the
$t$-expansions of $M$ where $t>1$.  Recall that $\kappa(M)<r(M)$ if
and only if $E(M)$ is the union of two proper flats of $M$.

\begin{thm}\label{thm:notmaxkappa}
  For a matroid $M$, the following statements are equivalent:
  \begin{itemize}
  \item[(1)] there are proper flats $F_1$ and $F_2$ of $M$ for which
    $|E(M)-(F_1\cup F_2)|\leq 1$,
  \item[(2)] $\kappa(M^t)<r(M^t)$ for all $t\geq 2$,
  \item[(3)] $\kappa(M^t)<r(M^t)$ for some $t\geq 2$.
  \end{itemize}
  Thus, if $E(M)$ is not the union of two proper flats of $M$ and a
  singleton subset of $E(M)$, then $\kappa(M^t)=r(M^t)$ for all
  $t\in\mathbb{N}$.
\end{thm}

\begin{proof}
  Assume that statement (1) holds.  If $E(M)=F_1\cup F_2$, then
  statement (2) follows from Theorem \ref{thm:vertkexp}, so assume
  instead that $E(M)-(F_1\cup F_2)=\{e\}$.  By Lemma
  \ref{lem:flatstructure1}, both $S_{F_1}$ and $S_{F_2}$ are flats of
  $M^t$.  Now $r(M^t)-r_{M^t}(S_{F_1})=t(r(M)-r_M(F_1))\geq t$, and
  likewise for $S_{F_2}$.  Since $S_{F_1}$ and $S_{F_2}$ are flats of
  $M^t$, if $\{A,B\}$ is a partition of $S_e$, then $S_{F_1}\cup A$
  and $S_{F_2}\cup B$ are proper flats of $M^t$ by Lemma
  \ref{lem:flatstructure3}, so statement (2) holds.

  Statement (2) obviously implies statement (3), so now assume that
  statement (3) holds.  We use $\theta(X)$ as it was defined in Lemma
  \ref{lem:flatstructure2}.  By statement (3), for some $t\geq 2$, the
  set $E(M^t)$ is the union of two proper flats of $M^t$; choose such
  flats $X_1$ and $X_2$ to maximize $|\theta(X_1)\cup \theta(X_2)|$.
  The elements of $X_1-S_{\theta(X_1)}$ are coloops of $M^t|X_1$, so
  we may assume that
  $(X_1-S_{\theta(X_1)})\cap S_{\theta(X_2)} = \emptyset$, and
  likewise that
  $(X_2-S_{\theta(X_2)})\cap S_{\theta(X_1)} = \emptyset$.  Now
  $E(M^t)-(S_{\theta(X_1)}\cup S_{\theta(X_2)})=S_D$ for some
  $D\subseteq E(M)$.  We claim that $|D|\leq 1$.  Assume instead that
  $|D|\geq 2$.  Thus, either $|X_1-S_{\theta(X_1)}|\geq t$ or
  $|X_2-S_{\theta(X_2)}|\geq t$; by symmetry, we may assume that
  $|X_1-S_{\theta(X_1)}|\geq t$.  Assume that
  $(X_1-S_{\theta(X_1)})\cap S_e\ne\emptyset$.  Fix
  $Z\subseteq X_1-(S_{\theta(X_1)}\cup S_e)$ with $|Z|=|S_e\cap X_2|$,
  and let $Y_1=\cl_{M^t}((X_1-Z)\cup S_e)$ and
  $Y_2=\cl_{M^t}((X_2-S_e)\cup Z)$.  The elements removed in each case
  were coloops of $M^t|X_1$ or $M^t|X_2$, so
  $r_{M^t}(Y_1)\leq r_{M^t}(X_1)$ and $r_{M^t}(Y_2)\leq r_{M^t}(X_2)$,
  so $Y_1$ and $Y_2$ are proper flats of $M^t$ with
  $E(M^t) = Y_1\cup Y_2$, and, since $e\in \theta(Y_1)-\theta(X_1)$,
  we have
  $|\theta(Y_1)\cup \theta(Y_2)|> |\theta(X_1)\cup \theta(X_2)|$,
  contrary to the choice of $X_1$ and $X_2$.  Thus, $|D|\leq 1$, so
  the proper flats $\theta(X_1)$ and $\theta(X_2)$ of $M$ show that
  statement (1) holds.
\end{proof}

\section{Branch-width}\label{section:branch}

Branch-width is considerably more delicate to work with, and, as we
will show, there is no counterpart of the equalities in Theorems
\ref{thm:connkexp} and \ref{thm:vertkexp}; what we prove instead is an
upper bound on $bw(M^t)$ in terms of $bw(M)$, the form of which is
similar to those results.  Also, we prove counterparts of Theorems
\ref{thm:Tuttegap}, \ref{thm:vertgap}, and \ref{thm:notmaxkappa}.

\begin{thm}\label{thm:bwbd}
  For a matroid $M$ and $t\in\mathbb{N}$, we have
  $bw(M^t)\leq t(bw(M)-1)+1$.
\end{thm}
	
\begin{proof}
  If $bw(M)=1$, then each element of $M$ is a loop or a coloop, so the
  same holds for $M^t$, so $bw(M^t)=1$, and the inequality holds in
  this case.  Now assume that $bw(M)>1$.

  Let $T_\phi$ be a branch-decomposition of $M$ of width $bw(M)$.  For
  each $a\in E(M)$, take a rooted tree with $t$ leaves where the
  degree of the root is $2$ and the degree of each other non-leaf is
  $3$, label its leaves with the elements of $S_a$, and identify the
  root with $\phi(a)$, the vertex of $T$ that is labeled $a$ in
  $T_\phi$.  This gives a branch-decomposition $T^t_{\phi'}$ of $M^t$.
  To prove the result, it suffices to show that
  $w(T^t_{\phi'})=t(bw(M)-1)+1$.

  If the edge $e$ of $T^t$ is an edge of $T$, then $e$
  displays $(S_Y,\overline{S_Y})$ for some $Y\subseteq E(M)$;
  otherwise $e$ displays $(X,\overline{X})$ for some $X\subsetneq S_a$
  and $a\in E(M)$.  In the first case,
  $$w_{M^t}(e) =\lambda_{M^t}(S_Y)+1 =t\cdot \lambda_M(Y)+1 =t\cdot
  (w_M(e)-1)+1.$$ Thus, $w_{M^t}(e)\leq t(bw(M)-1)+1$.  Also, equality
  holds for some edge of $T$.  In the second case, $r_{M^t}(X)< t$, so
  $w(e)\leq t\leq t(bw(M)-1)+1$.
\end{proof}

\begin{thm}\label{thm:bwgap}
  For any $n\in\mathbb{N}$, there are matroids $M'$ and $N'$ with no
  coloops and with the following properties: both are positroids, both
  are transversal matroids, they have the same configuration, and
  $bw(N')-bw(M')\geq n$.
\end{thm}

\begin{proof}
  Let $M$ and $N$ be the matroids in Figure \ref{fig:bwexample}.  We
  will show that $M'=M^{3n}$ and $N'=N^{3n}$ have the properties in
  the theorem.  It is easy to check that both $M$ and $N$ are
  positroids and transversal, and that they have the same
  configuration, so these properties hold for their $t$-expansions.

  After Theorem \ref{thm:tangle}, we showed that $bw(M)=3$, so
  $bw(M^t)\leq 2t+1$ by Theorem \ref{thm:bwbd}.  Equality will follow
  by showing that the set
  $\mathcal{T} = \{X\subseteq E(M^t)\,:\,r_{M^t}(X)<2t\}$ is a tangle
  of order $2t+1$. Properties (T1) and (T4) are immediate.  Let $L_1$,
  $L_2$, and $L_3$ be the $3$-point lines of $M$.  Note that a
  hyperplane of $M^t$ contains at most one of $S_{L_1}$, $S_{L_2}$,
  and $S_{L_3}$, and so, since $r(M^t)-r_{M^t}(S_{L_i})=t$, the
  largest hyperplanes of $M^t$ have $4t-1$ elements.  Therefore, since
  $|E(M^t)|=9t$, for any set $X\subseteq E(M^t)$, at least one of $X$
  and $\overline{X}$ spans $M^t$.  Property (T2) follows.  Also, since
  the smallest closures of circuits are $S_{L_1}$, $S_{L_2}$, and
  $S_{L_3}$, which have rank $2t$, if $X\in \mathcal{T}$, so
  $r_{M^t}(X)<2t$, then $X$ is independent, and so $|X|<2t$, from
  which property (T3) follows.  Thus, $bw(M^t)=2t+1$.
	
  We next show that $bw(N^t)=2t+\lceil t/3\rceil+1$.  We first give a
  branch-decomposition of $N^t$ that has width
  $2t+\lceil t/3\rceil+1$, thereby showing that
  $bw(N^t)\leq 2t+\lceil t/3\rceil+1$.  Let the cyclic lines of $N$ be
  $L_1$, $L_2$, and $L_3$ where $4\in L_1\cap L_2$.  Partition $S_1$
  into three sets, $X_1,X_2,X_3$, each of size $\lfloor t/3\rfloor$ or
  $\lceil t/3\rceil$.  Construct a branch-decomposition of $N^t$ for
  which, for some degree-$3$ vertex $v$ of the tree $T$, deleting $v$
  produces three subtrees with the labels on the leaves of these
  subtrees being $S_{L_1}\cup X_1$, $(S_{L_2}-S_4)\cup X_2$, and
  $S_{L_3}\cup X_3$.  Note that for the pair $(X,\overline{X})$ that
  is displayed by any edge of the tree, one of $X$ and $\overline{X}$
  spans $N^t$ and the other has rank at most $2t+\lceil t/3\rceil$, so
  $w(e) \leq 2t+\lceil t/3\rceil$; also, equality holds for at least
  one edge that is incident with $v$. Thus,
  $w(T)=2t+\lceil t/3\rceil+1$, so
  $bw(N^t)\leq 2t+\lceil t/3\rceil+1$.
	
  The equality $bw(N^t)=2t+\lceil t/3\rceil+1$ will follow by showing
  that
  $$\mathcal{T} = \{ X\subseteq E(N^t)\,:\,r_{N^t}(X) <2t+\lceil
  t/3\rceil\}$$ is a tangle of order $2t+\lceil
  t/3\rceil+1$. Properties (T1) and (T4) are immediate.  The same
  argument as used for $M^t$ shows that if $X\subseteq E(N^t)$, then
  either $X$ or $\overline{X}$ spans $N^t$, so property (T2) follows.
  Note that $X\in\mathcal{T}$ if and only if
  $\cl_{N^t}(X)\in\mathcal{T}$, so to prove property (T3), it suffices
  to show that no union of three flats, $X$, $Y$, and $Z$, of rank
  $2t+\lceil t/3\rceil-1$ can be $E(N^t)$.  Note that a flat $F$ of
  rank $2t+\lceil t/3\rceil-1$ is either independent or is the union
  of one of $S_{L_1}$, $S_{L_2}$, and $S_{L_3}$ and a set of
  $\lceil t/3\rceil-1$ coloops of $N^t|F$.  If at least one of $X$,
  $Y$, and $Z$ is independent, then
  $$|X\cup Y\cup Z|\leq 2t+\lceil t/3\rceil-1+ 2\bigl(3t+ \lceil
  t/3\rceil-1 \bigr)<9t.$$
  If none of $X$, $Y$, and $Z$ is
  independent, then since $S_{L_1}\cap S_{L_2} = S_4$, we have
  $$|X\cup Y\cup Z|\leq 3\bigl(3t+ \lceil
  t/3\rceil-1 \bigr)-t<9t.$$ In either case,
  $X\cup Y\cup Z\ne E(N^t)$, so property (T3) holds.

  Thus, $bw(N^{3n})-bw(M^{3n})= 2\cdot 3n+n+1 - (2\cdot 3n+1) = n$, as
  needed.
\end{proof}

The matroids $M$ and $N$ used in that proof show that equality can
hold in the inequality in Theorem \ref{thm:bwbd}, as in the case of
$M$, but, as in the case of $N$, it might be a strict inequality.
Also, in contrast to Theorem \ref{thm:vertkexp}, the inequality can be
strict even if the branch-width is less that the maximum that it can
be, namely, one more than the rank.  To see that, observe that
$bw(N^2)=6<r(N^2)+1=7$, but $bw(N^6)=15$ for the $3$-expansion $N^6$
of $N^2$.

The proof of the next result adapts the ideas in the proof of Theorem
\ref{thm:notmaxkappa}.

\begin{thm}
  For a matroid $M$, the following statements are equivalent:
  \begin{itemize}
  \item[(1)] there are proper flats $F_1$, $F_2$, and $F_3$ of $M$ for
    which $|E(M)-(F_1\cup F_2\cup F_3)|\leq 2$,
  \item[(2)] $bw(M^t)\leq r(M^t)$ for all $t\geq 3$,
  \item[(3)] $bw(M^t)\leq r(M^t)$ for some $t\geq 3$.
  \end{itemize}
  Thus, if $E(M)$ is not the union of three proper flats of $M$ and a
  $2$-element subset of $E(M)$, then $bw(M^t)=r(M^t)+1$ for all
  $t\in\mathbb{N}$.
\end{thm}

\begin{proof}
  Assume that statement (1) holds.  Let
  $D=E(M)-(F_1\cup F_2\cup F_3)$.  If $D=\emptyset$, then
  $E(M^t)=S_{F_1}\cup S_{F_2}\cup S_{F_3}$, so statement (2) follows
  by Lemma \ref{lem:hall}.  Now assume that $D\ne\emptyset$.  For each
  $i\in[3]$, the set $S_{F_i}$ is a flat of $M^t$ and
  $$r(M^t)-r_{M^t}(S_{F_i})=t(r(M)-r_M(F_i))\geq t.$$ 
  Since $|D|\leq 2$, for any $t\geq 3$, there is a partition
  $\{A_1,A_2,A_3\}$ of $S_D$ with $|A_i|<t$ for all $i\in[3]$.  By
  Lemma \ref{lem:flatstructure3}, the sets $S_{F_i}\cup A_i$, for
  $i\in[3]$, are flats of $M^t$, so statement (2) holds by Lemma
  \ref{lem:hall}.

  Statement (2) implies statement (3), so now assume that statement
  (3) holds.  We use $\theta(X)$ as in Lemma \ref{lem:flatstructure2}.
  By statement (3), for some $t\geq 3$, the set $E(M^t)$ is the union
  of three proper flats of $M^t$; pick such flats, $X_1$, $X_2$, and
  $X_3$, so that $|\theta(X_1)\cup \theta(X_2) \cup \theta(X_3)|$ is
  maximal.  As in the proof of Theorem \ref{thm:notmaxkappa}, we may
  assume that $(X_i-S_{\theta(X_i)})\cap S_{\theta(X_j)} = \emptyset$
  for all $\{i,j\}\subset [3]$.  Now
  $E(M^t)-(S_{\theta(X_1)}\cup S_{\theta(X_2)} \cup
  S_{\theta(X_3)})=S_D$ for some $D\subseteq E(M)$.  We claim that
  $|D|\leq 2$.  Assume instead that $|D|\geq 3$.  Thus,
  $|X_i-S_{\theta(X_i)}|\geq t$ for at least one $i\in[3]$; by
  symmetry, we may assume that $|X_1-S_{\theta(X_1)}|\geq t$.  Assume
  that $(X_1-S_{\theta(X_1)})\cap S_e\ne\emptyset$.  The exchange
  argument used in the proof of Theorem \ref{thm:notmaxkappa} can be
  adapted to move the elements of $S_e\cap( X_2\cup X_3)$ into $X_1$
  without increasing the ranks of the flats; this yields three proper
  flats that contradict the maximality assumption, so $|D|\leq 2$.
  Thus, the proper flats $\theta(X_1)$, $\theta(X_2)$, and
  $\theta(X_3)$ of $M$ show that statement (1) holds.
\end{proof}

\end{document}